\documentclass[10pt]{article}
\usepackage{amsmath,amssymb,amsthm,mathrsfs}
\theoremstyle{definition}
\newtheorem{theorem}{Theorem}[section]
\newtheorem{proposition}[theorem]{Proposition}
\newtheorem{lemma}[theorem]{Lemma}
\newtheorem{remark}[theorem]{Remark}

\newtheorem{corollary}[theorem]{Corollary}

\title{The space of tempered distributions as a k-space}
\author{Kei Harada \\
\begin{small}
Graduate School of Mathematics, Nagoya University
\end{small} \\
Hayato Saigo \\
\begin{small}
Research Institute for Mathematical Sciences, Kyoto University
\end{small}
}

\date{}

\begin{document}

\maketitle
\begin{abstract}
In this paper, we investigate the roles of compact sets in 
the space of tempered distributions $\mathscr{S}^{\prime}$.
The key notion is ``k-spaces'', which constitute a fairly general class of topological spaces. In a k-space, 
the system of compact sets controls continuous functions and Borel measures. 

Focusing on the k-space structure of $\mathscr{S}^{\prime}$, we prove some theorems 
which seem to be fundamental 
for infinite dimensional harmonic analysis from a new and unified view point. For example, 
the invariance principle of Donsker for the white noise measure 
is shown in terms of infinite dimansional characteristic functions.  
\end{abstract}

\section{Introduction}
Compactness is one of the most important 
notions in mathematics. While
compact sets are usually defined in terms of open sets, this does not mean that compactness is less fundamental than openness. 
To see this, let us take some examples from measure theory.

It is well known that there exist shift invariant measures, 
or Haar measures, on any locally compact group. 
A. Weil \cite{We} showed that the existence of 
shift invariant measure is, in a sense, 
equivalent to locally compactness. 
This means that the locally compactness is fundamental for harmonic analysis. 

The trouble is that most infinite dimensional spaces are not locally compact, but it is not the end of the story. 
Infinite dimensional calculus such as White Noise Analysis (\cite{HID})
makes use of the theory of Radon measures.
In this theory, the notion of inner regularity plays crucial roles. 
A Radon measure is a locally finite inner regular measure, 
and a finitely additive regular measure on a Radon space is
$\sigma $-additive if and only if it is inner regular 
(see \cite{ADA} for details). 
As inner regular measures are defined in terms of compact sets, 
it is natural to suppose consider that compactness is the most fundamental 
notion for measure theory. 
In fact, the theory of Radon measures can be defined by 
the notion of ``compactology'' introduced by Weil. 
More concretely, compactology defines a Radon measure as a system of mutually compatible measures associated 
to the system of compact sets (see e.g. \cite{SCH1}). 

On the other hand, there is a counterpart for compactology 
in the context of general topology, 
that is, the notion of k-space (compactly generated space) 
\cite{KEL}. 
A topological space $X$ is called a k-space if and only if any mapping  continuous on every compact subset is continuous, or equivalently, a subset $F$ in $X$ is closed if and only if $K \cap F$ is compact for any compact $K \subset X$. 
This means the topology is controlled by the system of compact sets. 
From the viewpoint of category theory, the category of Hausdorff k-spaces (CGHaus) has many good features (\cite{Mac}). 
For example, it admits the canonical structure for exponential objects (``function spaces''). 
Moreover, this category is natural to study and generalize Gelfand-type duality in analysis (\cite{DUB}, \cite{NEG}). 
Maclane even said that the category of Hausdorff k-spaces is ``right'', while 
the category of all topological spaces is ``wrong'' (\cite{Mac}).

Not all spaces, but most of practical spaces (as domain spaces) are k-spaces. 
It is known that
any locally compact space or any first-countable space (in particular Polish space) is 
a k-space. From the viewpoint of infinite dimensional analysis, 
it is natural to ask whether each space of distributions is k or not. 

In the present paper, we focus on the fact that the space of tempered distributions is a k-space 
and apply it to prove some notable theorems. 

\section{k-properties in the spaces of distributions}
The purpose of this section is to prove Theorem $\ref{252}$. 

\begin{lemma}
Let $H_1$, $H_2$ be Hilbert spaces, and let $i : H_1 \to H_2$ be a continuous linear map. If $F$ is a closed bounded convex subset of $H_1$, then $i(F)$ is closed in $H_2$. 
\end{lemma}

\begin{proof}
Let $y \in \overline{i(F)}$, then there exists a sequence $\{x_n\}$ in $F$ such that $i(x_n) \to y$.
Since $F$ is bounded in $H_1$, there exists a weakly convergent subsequence $x_{n_k} \to x_0$. 
Now, because the map $x \mapsto \langle i(x), z \rangle _{H_2}$ belongs to $H_1^{\ast}$ for all $z \in H_2$, we obtain 
\[
\langle i(x_0), z \rangle _{H_2} 
= \lim _{k \to \infty} \langle i(x_{n_k}), z \rangle _{H_2} 
= \langle y, z \rangle _{H_2},
\]
which shows $i(x_0) = y$. 
As $F$ is closed and convex, $x_0 \in F$, and hence we have $y \in i(F)$.    
\end{proof}

\begin{theorem}\label{252}
The space of tempered distributions $\mathscr{S}^{\prime}$ equipped with the strong topology is a k-space. 
\end{theorem}

\begin{proof}
As is well-known, $\mathscr{S}^{\prime}$ is the inductive limit of a sequence of Hilbert spaces $\{H_{-n}\}_{n=1}^{\infty}$. 
Let us denote by $i_n$ the canonical imbedding map $H_{-n} \to \mathscr{S}^{\prime}$.

Let $F \subset \mathscr{S}^{\prime}$ satisfy for all compact $K$, $K \cap F$ is closed.   
If $K_n \subset H_{-n}$ is compact, 
then $i_n(K_n)$ is compact, 
$F \cap i_n(K_n)$ is closed, 
and hence $i_n^{-1}(F) \cap K_n$ is closed. Since $H_{-n}$ is a k-space, we obtain $i_n^{-1} (F)$ is closed. 

Let $x \notin F$, then $x \in i_{m}(H_{-m})$ for some $m \in \mathbb{N}$. Since $x \notin F$ and $H_{-m}$ is a Hilbert space, there exists $r_m>0$ such that $\{i_m^{-1}(x) +B_m(r_m) \} \cap i_m^{-1}(F) = \emptyset$,  where we set $B_m(r_m) = \{y_m : y_m \in H_{-m}, \| y_m \|_{H_{-m}} \leq r_m \}$. 
Let $i_{m, m+2}$ denote the imbedding map of $H_{-m}$ into $H_{-m-2}$. By the previous lemma, we obtain $i_{m, m+2} (B_m(r_m))$
is closed, and by the fact that this imbedding is compact, $i_{m, m+2} (B_m(r_m))$ is compact. Again, since $i_{m, m+2} (B_m(r_m)) \cap i_{m+2}^{-1}(F)= \emptyset$, there exists $r_{m+2}>0$ such that \\
$\{ i_{m+2}^{-1}(x) + i_{m,m+2} (B_m(r_m)) + B_{m+2}(r_{m+2}) \} \cap i_{m+2}^{-1}(F) = \emptyset$, where \\
$B_{m+2}(r_{m+2}) = \{y_{m+2} : y_{m+2} \in H_{-m-2}, \| y_{m+2} \|_{H_{-m-2}} \leq r_{m+2} \}$. 

By the reputation of this process, we obtain $\{r_{m+2k} \}_{k=0}^{\infty}$ 
such that \\
$\{ i_{m+2k}^{-1}(x) + i_{m, m+2k}(B_m(r_m)) + \cdots + B_{m+2k}(r_{m+2k})
\} \cap i_{m+2k}^{-1}(F) = \emptyset$ for all $k$. 
Now, set $O \subset \mathscr{S}^{\prime}$ by 
\[
O = \cup _{k=0} ^{\infty} \{ i_m(B_m(r_m)) + i_{m+2}(B_{m+2}(r_{m+2}))+ \cdots 
i_{m+2k} (B_{m+2k}(r_{m+2k}))\} ,
\]
then by the definition of $\{r_{m+2k}\}$, it follows that $\{x+O \} \cap F= \emptyset$, and by the definition of the inductive limit topology of locally convex spaces, we have $O$ is a neighbourhood of $0$. And hence we obtain $F$ is closed in $\mathscr{S}^{\prime}$. 
\end{proof}

\begin{remark}
In the same way, one can also prove $\mathscr{S}^{\prime}$ is a sequential space. 
\end{remark}

\begin{remark}
This result itself is not newly obtained. Actually, it is shown that every Montel (DF)-space is sequential and hence k. See \cite{KAK}, \cite{WEB} for details. 
\end{remark}

\begin{proposition}
$\mathscr{S}^{\prime}$ equipped with weak topology is not a k-space. 
\end{proposition}

\begin{proof}
Assume the space is a k-space, then the weak topology is the finest topology under which $\mathscr{S}^{\prime}$ has the same compact sets as those in the weak topology. 
However, as is well known, $\mathscr{S}^{\prime}$ has the same compact sets under the weak topology and the strong topology.  
Hence by Theorem $\ref{252}$, these two topology coincides, which is a contradiction. 
\end{proof}

This may endorse that 
the strong topology is ``right'' for $\mathscr{S}^{\prime}$ as a domain space of infinite dimensional harmonic analysis. 
It is also known that $\mathscr{D}$ and $\mathscr{D}^{\prime}$ are not k-spaces (see \cite{SHI}, \cite{SMO}). 
So $\mathscr{D}^{\prime}$ is ``wrong'' compared to $\mathscr{S}^{\prime}$, at least from the viewpoint of 
categorical analysis for Gelfand-type duality (\cite{DUB}, \cite{NEG}).

\section{Prohorov's conditions }
In this section, as an application of the results in the previous section, we discuss Prohorov's theorem on the spaces of distributions. 
A topological space $X$ is said to satisfy Prohorov's condition ($P$) if 
any relatively compact subset of probability Borel measures is uniformly tight. 
We say that $X$ satisfies condition ($P^{\prime}$) if any relatively compact subset of signed Borel measures is uniformly tight (See \cite{DHF}). 
These conditions have much to do with the topological feature of the space. For instance, it is known that any Polish space satisfies condition $(P^{\prime})$. 

The following theorem is taken from \cite[Theorem 5.]{DHF}.
A topological space $X$ is said to be {\em hemicompact} if 
there exists an increasing sequence 
of compact subsets $\{K_n\}$ such that any compact subset $K$ is 
contained in some $K_n$. Such a sequence is said to be fundamental.  

\begin{theorem}
If a Radon space $X$ is a hemicompact k-space, 
then $X$ satisfies condition ($P^{\prime}$). 
\end{theorem}

\begin{lemma}\label{401}
Let $K$ be a compact subset in $\mathscr{S}^{\prime}$, 
then there exists $n \in \mathbb{N}$ 
such that $K = i_n(L)$ for some compact set $L \subset \mathscr{S}_{-n}$. 
\end{lemma}
\begin{proof}
Let us use the notations in Theorem $\ref{252}$. 
Since the absolute polar set of $K$ is a neighborhood of $0$, it contains the set $\{x \in \mathscr{S} ~| ~|x|_{H_n} <\delta  \}$ for some $n \in \mathbb{N}$ and $\delta >0$, and hence $K$ is contained in a bounded set in $H_{-m}$. 
Thus ${i_m}^{-1}(K)$ is bounded and closed. 
Since $i_{m,m+2}$ is 
compact, ${i_{m+2}}^{-1}(K)$ is compact. 
\end{proof}

\begin{theorem}\label{261}
$\mathscr{S}^{\prime}$ is hemicompact. 
\end{theorem}

\begin{proof}
Set $r_1=1$ and set $K_1 = i_1(B_1(r_1))$. 
For each $n \in \mathbb{N}$, set $r_n >0$ large enough so that 
$i_{k,n}(B_k(n)) \subset B_n(r_n)$ for all $1 \leq k \leq n$. 
Set $K_n = i_n(B_n(r_n))$, then it is immediate that $\cup K_n = \mathscr{S}^{\prime}$. 
And by the previous lemma, we obtain any compact set is contained in some $K_n$. 
\end{proof}

\begin{corollary}\label{290}
$\mathscr{S}^{\prime}$ satisfies condition $(P^{\prime})$, and hence condition $(P)$. 
\end{corollary}

\begin{remark} 
If $X$ is a Fr\'{e}chet space and is not locally compact, $X$ is never 
hemicompact. It is because if $X$ is $\sigma$-compact, 
we have contradiction by Baire category theorem. 
\end{remark}

\begin{remark}
It can be proved that $\mathscr{S}^{\prime}$ satisfies condition $(P)$ from a different point of view. (See \cite[Th\'{e}or\`{e}me I.6.5]{FER}).  
\end{remark}

\begin{proposition}
$\mathscr{D}^{\prime}$ is not hemicompact. 
\end{proposition}

\begin{proof}
Since $\mathscr{D}$ is Montel, 
$
\{f \in \mathscr{D}^{\prime} | ~
|\langle x, f \rangle| \leq 1, \forall x \in O  \}
$, 
the absolute polar set of an open set 
$O \subset \mathscr{D}$, is compact, and the absolute polar set of 
a compact subset in $\mathscr{D}^{\prime}$ is a neighborhood of $0$ in 
$\mathscr{D}$. 
Assume that there exists a fundamental sequence of compact sets. Then by taking absolute polar sets, there exists a sequence 
$\{O_n\}$ in $\mathscr{D}$ such that each $O_n$ is a neighborhood of $0$ and any neighborhood of $0$ contains some $O_n$. Hence it follows that $\mathscr{D}$ is first-countable, which is a contradiction. 
\end{proof}

Though $\mathscr{D}^{\prime}$ is not hemicompact,  
it satisfies condition $(P)$ because 
$\mathscr{D}$ is the strict inductive limit of 
a sequence of Fr\'{e}chet-Montel spaces (see \cite{FER}). 
But the methods of \cite{FER}, which make use of positivity, do not work well when measures are signed.

\section{Continuous functions}
In this section, we apply the results in Section 2 to the analysis of continuous functions. As a direct consequence, the structure of k-spaces gives a useful criterion 
for continuity.  
Let us take an example from 
White Noise Analysis \cite[Theorem 4.7.]{HID}. 
The property of k-spaces would help making the proof of the following theorem clearer because we only need to prove the continuity on each compact subset of $\mathscr{S}^{\prime}$. 
Actually, by Lemma $\ref{401}$, all we have to see is the continuity in each bounded set in $H_{-n}$.  

\begin{theorem}
Let $(\mathscr{S})$ be the space of Hida test functionals (\cite{HID}). 
Every $\varphi \in (\mathscr{S})$ has a unique pointwise defined, 
strongly continuous representative. 
\end{theorem}

The k-space structure of the domain space is also helpful for the analysis of the space of continuous functions. 
Let $C(\mathscr{S}^{\prime})$ be the space of continuous functions defined on $\mathscr{S}^{\prime}$. 
We will denote by $T_K$ the topology of uniform convergence on every compact set. Since $\mathscr{S}^{\prime}$ is a k-space, it follows that $T_K$ is complete. Furthermore, as $\mathscr{S}^{\prime}$ is hemicompact, the convergence on each $K_n$ is sufficient for $T_K$ convergence, where $\{K_n\}$ is a fundamental sequence. 
Hence this topology is metrizable. 
It also follows that $T_K$ is separable, because each $C(K_n)$ is separable. 
Summarizing, we have the following theorem.  
\begin{theorem}
The topological vector space $C(\mathscr{S}^{\prime})$ equipped with $T_K$ topology is a separable Fr\'{e}chet space. 
\end{theorem}

There is an 
analogous topology on the space of bounded continuous functions $C_b(X)$, so called 
$T_t$-topology, introduced by L. Le Cam \cite{LC}. 
It is known that if $X$ is a k-space, then $(C_b(X), T_t)$ is 
complete. If $X$ is a Radon space, then 
the dual of $(C_b(X), T_t)$ is $M(X)$, the space of bounded measures.  
If $X$ satisfies condition $(P^{\prime})$, then $T_t$ coincides with 
the Mackey topology $\tau \left( C_b(X), M(X) \right)$. 
See \cite{DHF}, \cite{LC} for details. 

\section{Infinite dimensional characteristic functions}

In this section, we discuss the characteristic functions of measures on $\mathscr{S}^{\prime}$. The following two results are known (see Fernique \cite{FER}), but it is worth pointing out that our proof is based on the property of k-spaces, while in \cite{FER} 
tensor products of nuclear spaces is used. 

For a probability measure $\mu$ on $\mathscr{S}^{\prime}$ and $\varphi \in \mathscr{S}$, we denote the characteristic function of $\mu$ by 
$\widehat{\mu}(\varphi) = \int_{\mathscr{S}^{\prime}} \exp\left( i \langle x, \varphi \rangle \right) d\mu (x)$. 

\begin{theorem}\label{420}
Let $\{ \mu _n \}$ be a sequence of probability measures on $\mathscr{S}^{\prime}$. Assume that $\widehat{\mu_n}(\varphi)$ converges 
for every $\varphi \in \mathscr{S}$. 
Then $\{ \mu_n\}$ converges weakly to some probability measure if and only if $\{ \widehat{\mu_n} \}$ is equicontinuous at $0$, 
that is, for all $\varepsilon >0$, there exist $m \in \mathbb{N}$ and $\delta>0$ such that $|\varphi|_m < \delta \Rightarrow 
|1 - \widehat{\mu_n}(\varphi) | < \varepsilon$ for all $n$. 
\end{theorem}

\begin{proof}
By Minlos' theorem, the equicontinuity of characteristic functions is equivalent to uniform tightness of probability measures (see e.g. \cite{BO}). 

If $\{ \mu _n\}$ converges weakly, $\{ \mu_n \}$ is a relatively compact subset, and hence by Corollary $\ref{290}$, $\{\mu _n\}$ is uniformly tight. 
Conversely, if $\{ \widehat{\mu_n} \}$ is equicontinuous at $0$, 
then $\{ \mu_n \}$ is a relatively compact subset. Since the limit of a subnet of $\{ \mu_n \}$ is uniquely determined by the characteristic function, $\{\mu _n\}$ is convergent. 
\end{proof}

\begin{theorem}\label{440}
Let $\{ \mu _n \}$ and $\mu$ be probability measures on $\mathscr{S}^{\prime}$. 
Then $\{ \mu _n \}$ converges to $ \mu $ weakly if and only if 
$\{ \widehat{\mu _n }(\varphi) \}$ converges to $\widehat{ \mu} (\varphi)$ for every $\varphi \in \mathscr{S}$. 
\end{theorem}

\begin{proof}
Let $\{ \widehat{\mu _n }(\varphi) \} \to \widehat{ \mu} (\varphi)$ for every $\varphi \in \mathscr{S}$. 
Assume that $\{ \widehat{\mu _n }(\cdot) \}$ is not equicontinuous 
at $0$, then there exist $\eta >0$, a sequence 
$\{ \varphi _k\} \subset \mathscr{S}$ and $\{ n_k \}$ such that 
$\varphi _k \to 0$ and 
\[
1 - Re(\widehat{\mu _{n_k}} (\varphi _k)) \geq \eta .
\]
Since $\{ \varphi _k \}$ converges to $0$, there exists a subsequence 
$\{ \varphi _{k_l} \}$ satisfying 
\[
\sum _{l =1} ^{\infty} |\varphi _ {k_l} |_p ^2 < \infty
\]
for all $p \in \mathbb{N}$. 
Now, set 
\[
F(x) = \exp \left( - \sum _{l=1}^{\infty} 
\langle \varphi _{k_l}, x \rangle ^2 \right)
\]
for $x \in \mathscr{S}^{\prime}$. By Proposition $\ref{401}$, $F$ is $T_K$-limit of positive definite continuous functions. Since $\mathscr{S}^{\prime}$ is a k-space, $F$ is positive definite continuous function. 

For any $\varepsilon >0$, by Lebesgue's dominated convergence theorem, 
there exists $l_0 \in \mathbb{N}$ such that 
\[
\int \left\{ 1- F_{l_0}(x) \right\} d\mu (x) 
< \varepsilon ,
\]
where 
\[
F_{l_0} (x) = \exp \left( - \sum _{l=l_0}^{\infty} 
\langle \varphi _{k_l}, x \rangle ^2  \right) .
\]
As $F_{l_0}$ is positive definite continuous function on $\mathscr{S}^{\prime}$, by Minlos' theorem, there exists a unique 
probability measure $m$ on $\mathscr{S}$ with 
$\widehat{m}(x) = F_{l_0}(x)$. 
By Fubini's theorem, 
\[
\int _{\mathscr{S}^{\prime}} \left\{ 1- F_{l_0}(x) \right\} 
d\mu _n (x) = 
\int _{\mathscr{S}} \left\{ 1- \widehat{\mu _n}(\varphi ) \right\} 
dm(\varphi) ,
\]
hence, for sufficiently large $n$, we obtain 
\[
\int \left\{ 1- F_{l_0}(x) \right\} 
d\mu _n (x) < 2\varepsilon .
\]
This shows, for $l\geq l_0$ and sufficiently large $n$, 
\[
\int \left\{ 1- \exp (- \langle \varphi _{k_l}, x \rangle ^2)  \right\}
d\mu _n(x) < 2\varepsilon , 
\]
which shows 
\[
1 - Re(\widehat{\mu _{n_{k_l}}}(\varphi_{k_l})) \leq 2M\varepsilon ,
\]
where 
\[
M = \sup _{u \in \mathbb{R}} \frac{1-\text{cos}u}{1-e^{-u^2}} .
\]
 Since $\varepsilon$ is arbitrary, this is a contradiction. Hence 
$\{ \widehat{\mu _n }(\cdot) \}$ is equicontinuous 
at $0$. By theorem $\ref{420}$, it follows that $\mu_n \to \mu$ weakly. 

The converse is obvious. 
\end{proof}

These two theorems lead us to an analogue of 
the invariance principle of Donsker (see e.g. \cite{KS}) for the White Noise measure. 

\begin{theorem}
Let $\{ \xi _j\} _{j= -\infty}^{\infty}$ be a sequence of independent 
identically distributed random variables with mean $0$ and variance $1$ 
defined on some probability space ($\Omega, \mathcal{F}, P$). 
Define a stochastic process $X$ by 
\[
X_t(\omega ) = \xi _{[t]} (\omega ) 
\]
and set 
\[
X_t^{(n)} (\omega ) = \sqrt{n}X_{nt}(\omega ). 
\]
Let $P_n$ be the probability measure on $\mathscr{S}^{\prime}$ 
induced by $X^{(n)}$. Then $\{P_n\}$ weakly converges to the white noise measure. 
\end{theorem}

\begin{proof}
Let us compute the characteristic function of $P_n$. For $\varphi \in \mathscr{S}$, 

\begin{equation*}
\begin{split}
\widehat{P_n}(\varphi ) &= 
E\left[ \exp ( i\langle X^{(n)}, \varphi \rangle ) \right ] \\
&= \prod _{j = - \infty} ^{\infty} 
E\left[ \exp \left( i \sqrt{n} \xi _j a_j^{(n)} \right) \right ] , 
\end{split}
\end{equation*}
where $a_j^{(n)}$ is set by 
\[
a_j^{(n)} = \int _{\frac{j}{n}} ^{\frac{j+1}{n}} \varphi (t)dt .
\]
Let $C$ denote the characteristic function of $\xi _j$, then $C$ is $C^2$ function with $C(0) =1$, $C^{\prime}(0) =0$, and $C^{\prime \prime} (0) = -1$. By Taylor's formula,  
\begin{equation*}
\begin{split}
\widehat{P_n}(\varphi ) &= 
\exp \left( \sum _{j=-\infty}^{\infty} \log \left( C\left( \sqrt{n} 
a_j^{(n)} \right) \right) \right) \\
&= \exp \left( \sum _{j=-\infty}^{\infty} \log \left( 
1+ \frac{1}{2}C^{\prime \prime} \left( \theta _j^{(n)} \sqrt{n} 
a_j^{(n)} \right) n {a_j^{(n)}}^2 \right) \right)
, 
\end{split}
\end{equation*}
where $0 < \theta _j^{(n)} < 1$. Since $\varphi \in \mathscr{S}$ and  $\sqrt{n}a_j^{(n)}$ uniformly converges to $0$, 
\[
\lim_{n \to \infty} \widehat{ P_n } (\varphi) = 
\exp \left( - \frac{1}{2} \int _{-\infty} ^{\infty} {\varphi (t)}^2 dt \right).
\] 
\end{proof}
\section*{Acknowledgements}
The authors wish to express their sincere gratitude to Prof. T. Hida
for his encouragement and interest in their research.
They are very grateful to Prof. I. Ojima for his helpful advices and comments.

\end{document}